\documentclass[a4paper,oneside,10pt]{article}%
\usepackage{amsmath}
\usepackage{amsfonts}
\usepackage{amssymb}
\usepackage{graphicx}
\usepackage[square,numbers,sort&compress]{natbib}%
\providecommand{\U}[1]{\protect \rule{.1in}{.1in}}

\newtheorem{theorem}{Theorem}[section]

\newtheorem{definition}[theorem]{Definition}

\newtheorem{example}[theorem]{Example}

\newtheorem{lemma}[theorem]{Lemma}

\newtheorem{remark}[theorem]{Remark}

\newenvironment{proof}[1][Proof]{\noindent \textbf{#1.} }{\  $\Box$}
\numberwithin{equation}{section}

\begin{document}

\title{Product space for two processes with independent increments under nonlinear expectations}
\author{Qiang Gao \thanks{School of Mathematics, Shandong University, Jinan, Shandong
250100, PR China. qianggao1990@163.com}
\and Mingshang Hu \thanks{Qilu Institute of Finance, Shandong University, Jinan,
Shandong 250100, PR China. humingshang@sdu.edu.cn. Research supported by NSF
(No. 11201262 and 11301068) and Shandong Province (No.BS2013SF020 and
ZR2014AP005) }
\and Xiaojun Ji \thanks{School of Mathematics, Shandong University, Jinan, Shandong
250100, PR China. yzxxn2009@163.com}
\and Guomin Liu \thanks{Qilu Institute of Finance, Shandong University, Jinan,
Shandong 250100, PR China. sduliuguomin@163.com. Gao, Hu, Ji and Liu's
research was partially supported by NSF (No. 10921101) and by the 111 Project
(No. B12023) }}
\maketitle

\textbf{Abstract}. In this paper, we consider the product space for two
processes with independent increments under nonlinear expectations. By
introducing a discretization method, we construct a nonlinear expectation
under which the given two processes can be seen as a new process with
independent increments.



\addcontentsline{toc}{section}{\hspace*{1.8em}Abstract}

\section{Introduction}

Peng \cite{P1, P2} introduced the notions of distribution and independence
under nonlinear expectation spaces. Under sublinear case, Peng \cite{P5}
obtained the corresponding central limit theorem for a sequence of i.i.d.
random vectors. The limit distribution is called $G$-normal distribution.
Based on this distribution, Peng \cite{P3, P4} gave the definition of
$G$-Brownian motion, which is a kind of process with stationary and
independent increments, and then discussed the It\^{o} stochastic analysis
with respect to $G$-Brownian motion.

It is well-known that the existence for a sequence of i.i.d. random vectors is
important for central limit theorem. In the nonlinear case, Peng \cite{P7}
introduced the product space technique to construct a sequence of i.i.d.
random vectors. But this product space technique does not hold in the
continuous time case. More precisely, let $(M_{t})_{t\geq0}$ and
$(N_{t})_{t\geq0}$ be two $d$-dimensional processes with independent
increments defined respectively on nonlinear expectation spaces $(\Omega
_{1},\mathcal{H}_{1},\mathbb{\hat{E}}_{1})$ and $(\Omega_{2},\mathcal{H}%
_{2},\mathbb{\hat{E}}_{2})$, we want to construct a $2d$-dimensional process
$(\tilde{M}_{t},\tilde{N}_{t})_{t\geq0}$ with independent increments defined
on a nonlinear expectation space $(\Omega,\mathcal{H},\mathbb{\hat{E}})$ such
that $(\tilde{M}_{t})_{t\geq0}\overset{d}{=}(M_{t})_{t\geq0}$ and $(\tilde
{N}_{t})_{t\geq0}\overset{d}{=}(N_{t})_{t\geq0}$. Usually, set $\Omega
=\Omega_{1}\times \Omega_{2}$, $\tilde{M}_{t}(\omega)=M_{t}(\omega_{1})$,
$\tilde{N}_{t}(\omega)=N_{t}(\omega_{2})$ for each $\omega=(\omega_{1}%
,\omega_{2})\in \Omega$, $t\geq0$. If we use Peng's product space technique,
then we can only get a $2d$-dimensional process $(\tilde{M}_{t},\tilde{N}%
_{t})_{t\geq0}$ such that $(\tilde{M}_{t})_{t\geq0}$ is independent from
$(\tilde{N}_{t})_{t\geq0}$ or $(\tilde{N}_{t})_{t\geq0}$ is independent from
$(\tilde{M}_{t})_{t\geq0}$. Different from linear expectation case, the
independence is not mutual under nonlinear case (see \cite{H1}). So this
$(\tilde{M}_{t},\tilde{N}_{t})_{t\geq0}$ is not a process with independent increments.

In this paper, we introduce a discretization method, which can overcome the
problem of independence. More precisely, for each given $\mathcal{D}%
_{n}=\{i2^{-n}:i\geq0\}$, we can construct a nonlinear expectation
$\mathbb{\hat{E}}^{n}$ under which $(\tilde{M}_{t},\tilde{N}_{t}%
)_{t\in \mathcal{D}_{n}}$ posesses independent increments. But $\mathbb{\hat
{E}}^{n}$, $n\geq1$, are not consistent, i.e., the values of the same random
variable under $\mathbb{\hat{E}}^{n}$ are not equal. Fortunately, we can prove
that the limit of $\mathbb{\hat{E}}^{n}$ exisits by using the notion of
tightness, which was introduced by Peng in \cite{P8} to prove central limit
theorem under sublinear case. Denote the limit of $\mathbb{\hat{E}}^{n}$ by
$\mathbb{\hat{E}}$, we show that $(\tilde{M}_{t},\tilde{N}_{t})_{t\geq0}$ is
the process with independent increments under $\mathbb{\hat{E}}$.

This paper is organized as follows: In Section 2, we recall some basic notions
and results of nonlinear expectations. The main theorem is stated and proved
in Section 3.

\section{Preliminaries}

We present some basic notions and results of nonlinear and sublinear
expectations in this section. More details can be found in [1-20].

Let $\Omega$ be a given nonempty set and $\mathcal{H}$ be a linear space of
real-valued functions on $\Omega$ such that if $X_{1}$,$\dots$,$X_{d}%
\in \mathcal{H}$, then $\varphi(X_{1},X_{2},\dots,X_{d})\in \mathcal{H}$ for
each $\varphi \in C_{b.Lip}(\mathbb{R}^{d})$, where $C_{b.Lip}(\mathbb{R}^{d})$
denotes the set of all bounded and Lipschitz functions on $\mathbb{R}^{d}$.
$\mathcal{H}$ is considered as the space of random variables. Similarly,
$\{X=(X_{1},\dots,X_{d}):X_{i}\in \mathcal{H},i\leq d\}$ denotes the space of
$d$-dimensional random vectors.

\begin{definition}
A sublinear expectation $\hat{\mathbb{E}}$ on $\mathcal{H}$ is a functional
$\mathbb{\hat{E}}:\mathcal{H}\rightarrow \mathbb{R}$ satisfying the following
properties: for each $X,Y\in \mathcal{H}$,

\begin{description}
\item[(i)] \textbf{Monotonicity:}\quad$\mathbb{\hat{E}}[X]\geq \mathbb{\hat{E}%
}[Y]\  \  \text{if}\ X\geq Y$;

\item[(ii)] \textbf{Constant preserving:}\quad$\mathbb{\hat{E}}%
[c]=c\  \  \  \text{for}\ c\in \mathbb{R}$;

\item[(iii)] \textbf{Sub-additivity:}\quad$\mathbb{\hat{E}}[X+Y]\leq
\mathbb{\hat{E}}[X]+\mathbb{\hat{E}}[Y]$;

\item[(iv)] \textbf{Positive homogeneity:}\quad$\mathbb{\hat{E}}[\lambda
X]=\lambda \mathbb{\hat{E}}[X]\  \  \  \text{for}\  \lambda \geq0$.
\end{description}

The triple $(\Omega,\mathcal{H},\mathbb{\hat{E}})$ is called a sublinear
expectation space. If (i) and (ii) are satisfied, $\mathbb{\hat{E}}$ is called
a nonlinear expectation and the triple $(\Omega,\mathcal{H},\mathbb{\hat{E}})$
is called a nonlinear expectation space.
\end{definition}

Let $(\Omega,\mathcal{H},\mathbb{\hat{E}})$ be a nonlinear (resp. sublinear)
expectation space. For each given $d$-dimensional random vector $X$, we define
a functional on $C_{b.Lip}(\mathbb{R}^{d})$ by%
\[
\mathbb{\hat{F}}_{X}[\varphi]:=\mathbb{\hat{E}}[\varphi(X)]\text{ for each
}\varphi \in C_{b.Lip}(\mathbb{R}^{d}).
\]
It is easy to verify that $(\mathbb{R}^{d},C_{b.Lip}(\mathbb{R}^{d}%
),\mathbb{\hat{F}}_{X})$ forms a nonlinear (resp. sublinear) expectation
space. $\mathbb{\hat{F}}_{X}$ is called the distribution of $X$. Two
$d$-dimensional random vectors $X_{1}$ and $X_{2}$ defined respectively on
nonlinear expectation spaces $(\Omega_{1},\mathcal{H}_{1},\mathbb{\hat{E}}%
_{1})$ and $(\Omega_{2},\mathcal{H}_{2},\mathbb{\hat{E}}_{2})$ are called
identically distributed, denoted by $X_{1}\overset{d}{=}X_{2}$, if
$\mathbb{\hat{F}}_{X_{1}}=\mathbb{\hat{F}}_{X_{2}}$, i.e.,%
\[
\mathbb{\hat{E}}_{1}[\varphi(X_{1})]=\mathbb{\hat{E}}_{2}[\varphi
(X_{2})]\text{ for each }\varphi \in C_{b.Lip}(\mathbb{R}^{d}).
\]

Similar to the classical case, Peng \cite{P8} gave the following definition of
convergence in distribution.

\begin{definition}
Let $X_{n}$, $n\geq1$, be a sequence of $d$-dimensional random vectors defined
respectively on nonlinear (resp. sublinear) expectation spaces $(\Omega
_{n},\mathcal{H}_{n},\mathbb{\hat{E}}_{n})$. $\{X_{n}:n\geq1\}$ is said to
converge in distribution if, for each fixed $\varphi \in C_{b.Lip}%
(\mathbb{R}^{d})$, $\{ \mathbb{\hat{F}}_{X_{n}}[\varphi]:n\geq1\}$ is a Cauchy
sequence. Define%
\[
\mathbb{\hat{F}}[\varphi]=\lim_{n\rightarrow \infty}\mathbb{\hat{F}}_{X_{n}%
}[\varphi],
\]
the triple $(\mathbb{R}^{d},C_{b.Lip}(\mathbb{R}^{d}),\mathbb{\hat{F}})$ forms
a nonlinear (resp. sublinear) expectation space.
\end{definition}

If $X_{n}$, $n\geq1$, is a sequence of $d$-dimensional random vectors defined
on the same sublinear expectation space $(\Omega,\mathcal{H},\mathbb{\hat{E}%
})$ satisfying%
\[
\lim_{n,m\rightarrow \infty}\mathbb{\hat{E}}[|X_{n}-X_{m}|]=0,
\]
then we can deduce that $\{X_{n}:n\geq1\}$ converges in distribution by%
\[
|\mathbb{\hat{F}}_{X_{n}}[\varphi]-\mathbb{\hat{F}}_{X_{m}}[\varphi
]|=|\mathbb{\hat{E}}[\varphi(X_{n})]-\mathbb{\hat{E}}[\varphi(X_{m})]|\leq
C_{\varphi}\mathbb{\hat{E}}[|X_{n}-X_{m}|],
\]
where $C_{\varphi}$ is the Lipschitz constant of $\varphi$.

The following definition of tightness is important for obtaining a subsequence
which converges in distribution.

\begin{definition}
Let $X$ be a $d$-dimensional random vector defined on a sublinear expectation
space $(\Omega,\mathcal{H},\mathbb{\hat{E}})$. The distribution of $X$ is
called tight if, for each $\varepsilon>0$, there exists an $N>0$ and
$\varphi \in C_{b.Lip}(\mathbb{R}^{d})$ with $I_{\{x:|x|\geq N\}}\leq \varphi$
such that $\mathbb{\hat{F}}_{X}[\varphi]=\mathbb{\hat{E}}[\varphi
(X)]<\varepsilon$.
\end{definition}

\begin{definition}
Let $\{ \mathbb{\hat{E}}_{\lambda}:\lambda \in I\}$ be a family of nonlinear
expectations and $\mathbb{\hat{E}}$ be a sublinear expectation defined on
$(\Omega,\mathcal{H})$. $\{ \mathbb{\hat{E}}_{\lambda}:\lambda \in I\}$ is said
to be dominated by $\mathbb{\hat{E}}$ if, for each $\lambda \in I$,%
\[
\mathbb{\hat{E}}_{\lambda}[X]-\mathbb{\hat{E}}_{\lambda}[Y]\leq \mathbb{\hat
{E}}[X-Y]\text{ for each }X,Y\in \mathcal{H}.
\]

\end{definition}

\begin{definition}
Let $X_{\lambda}$, $\lambda \in I$, be a family of $d$-dimensional random
vectors defined respectively on nonlinear expectation spaces $(\Omega
_{\lambda},\mathcal{H}_{\lambda},\mathbb{\hat{E}}_{\lambda})$. $\{
\mathbb{\hat{F}}_{X_{\lambda}}:\lambda \in I\}$ is called tight if there exists
a tight sublinear expectation $\mathbb{\hat{F}}$ on $(\mathbb{R}^{d}%
,C_{b.Lip}(\mathbb{R}^{d}))$ which dominates $\{ \mathbb{\hat{F}}_{X_{\lambda
}}:\lambda \in I\}$.
\end{definition}

\begin{remark}
A family of sublinear expectations $\{ \mathbb{\hat{F}}_{X_{\lambda}}%
:\lambda \in I\}$ on $(\mathbb{R}^{d},C_{b.Lip}(\mathbb{R}^{d}))$ is tight if
and only if $\mathbb{\hat{F}}[\varphi]=\sup_{\lambda \in I}\mathbb{\hat{F}%
}_{X_{\lambda}}[\varphi]$ for each $\varphi \in C_{b.Lip}(\mathbb{R}^{d})$ is a
tight sublinear expectation.
\end{remark}

\begin{theorem}
\label{th2.1} (\cite{P8}) Let $X_{n}$, $n\geq1$, be a sequence of
$d$-dimensional random vectors defined respectively on nonlinear expectation
spaces $(\Omega_{n},\mathcal{H}_{n},\mathbb{\hat{E}}_{n})$. If $\{ \mathbb{\hat
{F}}_{X_{n}}:n\geq1\}$ is tight, then there exists a subsequence $\{X_{n_{i}%
}:i\geq1\}$ which converges in distribution.
\end{theorem}

The following definition of independence is fundamental in nonlinear
expectation theory.

\begin{definition}
Let $(\Omega,\mathcal{H},\mathbb{\hat{E}})$ be a nonlinear expectation space.
A $d$-dimensional random vector $Y$ is said to be independent from another
$m$-dimensional random vector $X$ under $\mathbb{\hat{E}}[\cdot]$ if, for each
test function $\varphi \in C_{b.Lip}(\mathbb{R}^{m+d})$, we have%
\[
\mathbb{\hat{E}}[\varphi(X,Y)]=\mathbb{\hat{E}}[\mathbb{\hat{E}}%
[\varphi(x,Y)]_{x=X}].
\]

\end{definition}

\begin{remark}
It is important to note that \textquotedblleft$Y$ is independent from $X$"
does not imply that \textquotedblleft$X$ is independent from $Y$" (see
\cite{H1}).
\end{remark}

A family of $d$-dimensional random vectors $(X_{t})_{t\geq0}$ on the same
nonlinear expectation space $(\Omega,\mathcal{H},\mathbb{\hat{E}})$ is called
a $d$-dimensional stochastic process.

\begin{definition}
Two $d$-dimensional processes $(X_{t})_{t\geq0}$ and $(Y_{t})_{t\geq0}$
defined respectively on nonlinear expectation spaces $(\Omega_{1}%
,\mathcal{H}_{1},\mathbb{\hat{E}}_{1})$ and $(\Omega_{2},\mathcal{H}%
_{2},\mathbb{\hat{E}}_{2})$ are called identically distributed, denoted by
$(X_{t})_{t\geq0}\overset{d}{=}(Y_{t})_{t\geq0}$, if for each $n\in \mathbb{N}%
$, $0\leq t_{1}<\cdots<t_{n}$, $(X_{t_{1}},\ldots,X_{t_{n}})\overset{d}%
{=}(Y_{t_{1}},\ldots,Y_{t_{n}})$, i.e.,%
\[
\mathbb{\hat{E}}_{1}[\varphi(X_{t_{1}},\ldots,X_{t_{n}})]=\mathbb{\hat{E}}%
_{2}[\varphi(Y_{t_{1}},\ldots,Y_{t_{n}})]\text{ for each }\varphi \in
C_{b.Lip}(\mathbb{R}^{n\times d}).
\]

\end{definition}

\begin{definition}
A $d$-dimensional process $(X_{t})_{t\geq0}$ with $X_{0}=0$ on a nonlinear
expectation space $(\Omega,\mathcal{H},\mathbb{\hat{E}})$ is said to have
independent increments if, for each $0\leq t_{1}<\cdots<t_{n}$, $X_{t_{n}%
}-X_{t_{n-1}}$ is independent from $(X_{t_{1}},\ldots,X_{t_{n-1}})$. A
$d$-dimensional process $(X_{t})_{t\geq0}$ with $X_{0}=0$ is said to have
stationary increments if, for each $t$, $s\geq0$, $X_{t+s}-X_{s}\overset{d}%
{=}X_{t}$.
\end{definition}

We give a typical example of process with stationary and independent increments.

\begin{example}
Let $\Gamma$ be a given bounded subset in $\mathbb{R}^{d\times d}$, where
$\mathbb{R}^{d\times d}$ denotes the set of all $d\times d$ matrices. Define
$G:\mathbb{S}(d)\rightarrow \mathbb{R}$ by%
\[
G(A)=\frac{1}{2}\sup_{Q\in \Gamma}\mathrm{tr}[AQQ^{T}]\text{ for each }%
A\in \mathbb{S}(d),
\]
where $\mathbb{S}(d)$ denotes the set of all $d\times d$ symmetric matrices. A
$d$-dimensional process $(B_{t})_{t\geq0}$ on a sublinear expectation space
$(\Omega,\mathcal{H},\mathbb{\hat{E}})$ is called a $G$-Brownian motion if the
following properties are satisfied:

\begin{description}
\item[(1)] $B_{0}=0$;

\item[(2)] It is a process with independent increments;

\item[(3)] For each $t$, $s\geq0$, $\mathbb{\hat{E}}[\varphi(B_{t+s}%
-B_{s})]=u^{\varphi}(t,0)$ for each $\varphi \in C_{b.Lip}(\mathbb{R}^{d})$,
where $u^{\varphi}$ is the viscosity solution of the following $G$-heat
equation:%
\[
\left \{
\begin{array}
[c]{l}%
\partial_{t}u(t,x)-G(D_{x}^{2}u(t,x))=0,\\
u(0,x)=\varphi(x).
\end{array}
\right.
\]

\end{description}

Obviously, (3) implies that the process $(B_{t})_{t\geq0}$ has stationary increments.
\end{example}

\section{Main result}

Let $(M_{t})_{t\geq0}$ and $(N_{t})_{t\geq0}$ be two $d$-dimensional processes
with independent increments defined respectively on nonlinear (resp.
sublinear) expectation spaces $(\Omega_{1},\mathcal{H}_{1},\mathbb{\hat{E}%
}_{1})$ and $(\Omega_{2},\mathcal{H}_{2},\mathbb{\hat{E}}_{2})$. We need the
following assumption:

\begin{description}
\item[(A)] There exist two sublinear expectations $\mathbb{\tilde{E}}%
_{1}:\mathcal{H}_{1}\rightarrow \mathbb{R}$ and $\mathbb{\tilde{E}}%
_{2}:\mathcal{H}_{2}\rightarrow \mathbb{R}$ satisfying:

\item  (1) $\mathbb{\tilde{E}}_{1}$ and $\mathbb{\tilde{E}}_{2}$ dominate
$\mathbb{\hat{E}}_{1}$ and $\mathbb{\hat{E}}_{2}$ respectively;

\item  (2) For each $t\geq0$,%
\[
\lim_{s\rightarrow t}(\mathbb{\tilde{E}}_{1}[|M_{s}-M_{t}|]+\mathbb{\tilde{E}%
}_{2}[|N_{s}-N_{t}|])=0.
\]

\end{description}

\begin{remark}
If $\mathbb{\hat{E}}_{1}$ and $\mathbb{\hat{E}}_{2}$ are sublinear
expectations, then we can get
\[
\lim_{s\rightarrow t}(\mathbb{\hat{E}}_{1}[|M_{s}-M_{t}|]+\mathbb{\hat{E}}%
_{2}[|N_{s}-N_{t}|])=0
\]
by $\mathbb{\hat{E}}_{1}[\cdot]\leq \mathbb{\tilde{E}}_{1}[\cdot]$ and
$\mathbb{\hat{E}}_{2}[\cdot]\leq \mathbb{\tilde{E}}_{2}[\cdot]$. So we can
replace $\mathbb{\tilde{E}}_{1}$ and $\mathbb{\tilde{E}}_{2}$ by
$\mathbb{\hat{E}}_{1}$ and $\mathbb{\hat{E}}_{2}$.
\end{remark}

Now we give our main theorem.

\begin{theorem}
\label{th3.1}Let $(M_{t})_{t\geq0}$ and $(N_{t})_{t\geq0}$ be two
$d$-dimensional processes with independent increments defined respectively on
nonlinear (resp. sublinear) expectation spaces $(\Omega_{1},\mathcal{H}%
_{1},\mathbb{\hat{E}}_{1})$ and $(\Omega_{2},\mathcal{H}_{2},\mathbb{\hat{E}%
}_{2})$ satisfying the assumption (A). Then there exists a $2d$-dimensional
process $(\tilde{M}_{t},\tilde{N}_{t})_{t\geq0}$ with independent increments
defined on a nonlinear (resp. sublinear) expectation space $(\Omega
,\mathcal{H},\mathbb{\hat{E}})$ such that $(\tilde{M}_{t})_{t\geq0}\overset
{d}{=}(M_{t})_{t\geq0}$ and $(\tilde{N}_{t})_{t\geq0}\overset{d}{=}%
(N_{t})_{t\geq0}$. Furthermore, $(\tilde{M}_{t},\tilde{N}_{t})_{t\geq0}$ is a
process with stationary and independent increments if $(M_{t})_{t\geq0}$ and
$(N_{t})_{t\geq0}$ are two processes with stationary and independent increments.
\end{theorem}

In the following, we only prove the sublinear expectation case. The nonlinear
expectation case can be proved by the same method. Moreover, the following
lemma shows that we only need to prove the theorem for $t\in \lbrack0,1]$.

\begin{lemma}
Let $(X_{t}^{i},Y_{t}^{i})_{t\in \lbrack0,1]}$, $i\geq0$, be a sequence of
$2d$-dimensional processes with independent increments defined respectively on
sublinear expectation spaces $(\bar{\Omega}_{i},\mathcal{\bar{H}}%
_{i},\mathbb{\bar{E}}_{i})$ such that $(X_{t}^{i})_{t\in \lbrack0,1]}%
\overset{d}{=}(M_{i+t}-M_{i})_{t\in \lbrack0,1]}$ and $(Y_{t}^{i})_{t\in
\lbrack0,1]}\overset{d}{=}(N_{i+t}-N_{i})_{t\in \lbrack0,1]}$. Then there
exists a $2d$-dimensional process $(\tilde{M}_{t},\tilde{N}_{t})_{t\geq0}$
with independent increments defined on a sublinear expectation space
$(\Omega,\mathcal{H},\mathbb{\hat{E}})$ such that $(\tilde{M}_{t})_{t\geq
0}\overset{d}{=}(M_{t})_{t\geq0}$ and $(\tilde{N}_{t})_{t\geq0}\overset{d}%
{=}(N_{t})_{t\geq0}$.
\end{lemma}

\begin{proof}
Set $(\Omega,\mathcal{H},\mathbb{\hat{E}})=(\Pi_{i=0}^{\infty}\bar{\Omega}%
_{i},\otimes_{i=0}^{\infty}\mathcal{\bar{H}}_{i},\otimes_{i=0}^{\infty
}\mathbb{\bar{E}}_{i})$ which is the product space of $\{(\bar{\Omega}%
_{i},\mathcal{\bar{H}}_{i},\mathbb{\bar{E}}_{i}):i\geq0\}$ (see \cite{P7}).
For each $\omega=(\omega_{i})_{i=0}^{\infty}$, define%
\[
\tilde{M}_{t}(\omega)=\sum_{i=0}^{[t]-1}X_{1}^{i}(\omega_{i})+X_{t-[t]}%
^{[t]}(\omega_{\lbrack t]}),\text{ }\tilde{N}_{t}(\omega)=\sum_{i=0}%
^{[t]-1}Y_{1}^{i}(\omega_{i})+Y_{t-[t]}^{[t]}(\omega_{\lbrack t]}).
\]
By Proposition 3.15 in Chapter I in \cite{P7}, we can easily obtain that
$(\tilde{M}_{t},\tilde{N}_{t})_{t\geq0}$ has independent increments property,
$(\tilde{M}_{t})_{t\geq0}\overset{d}{=}(M_{t})_{t\geq0}$ and $(\tilde{N}%
_{t})_{t\geq0}\overset{d}{=}(N_{t})_{t\geq0}$.
\end{proof}

Set $\Omega=\Omega_{1}\times \Omega_{2}=\{ \omega=(\omega_{1},\omega
_{2}):\omega_{1}\in \Omega_{1},\omega_{2}\in \Omega_{2}\}$. For each
$\omega=(\omega_{1},\omega_{2})\in \Omega$, define%
\[
\tilde{M}_{t}(\omega)=M_{t}(\omega_{1}),\text{ }\tilde{N}_{t}(\omega
)=N_{t}(\omega_{2})\text{ for }t\in \lbrack0,1].
\]
For notation simplicity, we denote $X_{t}=(\tilde{M}_{t},\tilde{N}_{t})$.
Define the space of random variables as follows:%
\[%
\begin{array}
[c]{r}%
\mathcal{H}=\{ \varphi(X_{t_{1}},X_{t_{2}}-X_{t_{1}},\ldots,X_{t_{n}%
}-X_{t_{n-1}}):\forall n\geq1,\forall0\leq t_{1}<t_{2}<\cdots<t_{n}\leq1,\\
\forall \varphi \in C_{b.Lip}(\mathbb{R}^{n\times2d})\}.
\end{array}
\]
In the following, we will construct a sublinear expectation $\mathbb{\hat{E}%
}:\mathcal{H}\rightarrow \mathbb{R}$ such that $(\tilde{M}_{t})_{t\in
\lbrack0,1]}\overset{d}{=}(M_{t})_{t\in \lbrack0,1]}$, $(\tilde{N}_{t}%
)_{t\in \lbrack0,1]}\overset{d}{=}(N_{t})_{t\in \lbrack0,1]}$ and $(\tilde
{M}_{t},\tilde{N}_{t})_{t\in \lbrack0,1]}$ possessing independent increments.
In order to construct $\mathbb{\hat{E}}$, we set, for each fixed $n\geq1$,%
\[
\mathcal{H}^{n}=\{ \varphi(X_{\delta_{n}},X_{2\delta_{n}}-X_{\delta_{n}%
},\ldots,X_{2^{n}\delta_{n}}-X_{(2^{n}-1)\delta_{n}}):\forall \varphi \in
C_{b.Lip}(\mathbb{R}^{2^{n}\times2d})\},
\]
where $\delta_{n}=2^{-n}$. Define $\mathbb{\hat{E}}^{n}:\mathcal{H}%
^{n}\rightarrow \mathbb{R}$ as follows:

Step 1. For each given $\phi(X_{k\delta_{n}}-X_{(k-1)\delta_{n}})=\phi
(\tilde{M}_{k\delta_{n}}-\tilde{M}_{(k-1)\delta_{n}},\tilde{N}_{k\delta_{n}%
}-\tilde{N}_{(k-1)\delta_{n}})\in \mathcal{H}^{n}$ with $k\leq2^{n}$ and
$\phi \in C_{b.Lip}(\mathbb{R}^{2d})$, define%
\[
\mathbb{\hat{E}}^{n}[\phi(X_{k\delta_{n}}-X_{(k-1)\delta_{n}})]=\mathbb{\hat
{E}}_{1}[\psi(M_{k\delta_{n}}-M_{(k-1)\delta_{n}})],
\]
where%
\[
\psi(x)=\mathbb{\hat{E}}_{2}[\phi(x,N_{k\delta_{n}}-N_{(k-1)\delta_{n}%
})]\text{ for each }x\in \mathbb{R}^{d}.
\]

Step 2. For each given $\varphi(X_{\delta_{n}},X_{2\delta_{n}}-X_{\delta_{n}%
},\ldots,X_{2^{n}\delta_{n}}-X_{(2^{n}-1)\delta_{n}})\in \mathcal{H}^{n}$ with
$\varphi \in C_{b.Lip}(\mathbb{R}^{2^{n}\times2d})$, define%
\[
\mathbb{\hat{E}}^{n}[\varphi(X_{\delta_{n}},X_{2\delta_{n}}-X_{\delta_{n}%
},\ldots,X_{2^{n}\delta_{n}}-X_{(2^{n}-1)\delta_{n}})]=\varphi_{0},
\]
where $\varphi_{0}$ is obtained backwardly by Step 1 in the following sense:
\[
\varphi_{2^{n}-1}(x_{1},x_{2},\ldots,x_{2^{n}-1})=\mathbb{\hat{E}}^{n}%
[\varphi(x_{1},x_{2},\ldots,x_{2^{n}-1},X_{2^{n}\delta_{n}}-X_{(2^{n}%
-1)\delta_{n}})],
\]%
\[
\varphi_{2^{n}-2}(x_{1},x_{2},\ldots,x_{2^{n}-2})=\mathbb{\hat{E}}^{n}%
[\varphi_{2^{n}-1}(x_{1},x_{2},\ldots,x_{2^{n}-2},X_{(2^{n}-1)\delta_{n}%
}-X_{(2^{n}-2)\delta_{n}})],
\]%
\[
\vdots
\]%
\[
\varphi_{1}(x_{1})=\mathbb{\hat{E}}^{n}[\varphi_{2}(x_{1},X_{2\delta_{n}%
}-X_{\delta_{n}})],
\]%
\[
\varphi_{0}=\mathbb{\hat{E}}^{n}[\varphi_{1}(X_{\delta_{n}})].
\]

\begin{lemma}
\label{le3.2}Let $(\Omega,\mathcal{H}^{n},\mathbb{\hat{E}}^{n})$ be defined as
above. Then

\begin{description}
\item[(1)] $(\Omega,\mathcal{H}^{n},\mathbb{\hat{E}}^{n})$ forms a sublinear
expectation space;

\item[(2)] For each $2\leq k\leq2^{n}$, $X_{k\delta_{n}}-X_{(k-1)\delta_{n}}$
is independent from $(X_{\delta_{n}},\ldots,X_{(k-1)\delta_{n}}-X_{(k-2)\delta
_{n}})$;

\item[(3)] $(\tilde{M}_{\delta_{n}},\tilde{M}_{2\delta_{n}}-\tilde{M}%
_{\delta_{n}},\ldots,\tilde{M}_{2^{n}\delta_{n}}-\tilde{M}_{(2^{n}%
-1)\delta_{n}})\overset{d}{=}(M_{\delta_{n}},M_{2\delta_{n}}-M_{\delta_{n}%
},\ldots,M_{2^{n}\delta_{n}}-M_{(2^{n}-1)\delta_{n}})$ and $(\tilde{N}%
_{\delta_{n}},\tilde{N}_{2\delta_{n}}-\tilde{N}_{\delta_{n}},\ldots,\tilde
{N}_{2^{n}\delta_{n}}-\tilde{N}_{(2^{n}-1)\delta_{n}})\overset{d}{=}%
(N_{\delta_{n}},N_{2\delta_{n}}-N_{\delta_{n}},\ldots,N_{2^{n}\delta_{n}%
}-N_{(2^{n}-1)\delta_{n}})$.
\end{description}
\end{lemma}

\begin{proof}
(1) It is easy to check that $\mathbb{\hat{E}}^{n}:\mathcal{H}^{n}%
\rightarrow \mathbb{R}$ is well-defined. We only prove $\mathbb{\hat{E}}^{n}$
satisfies monotonicity, the other properties can be similarly obtained. For
each given $Y=\varphi_{1}(X_{\delta_{n}},X_{2\delta_{n}}-X_{\delta_{n}}%
,\ldots,X_{2^{n}\delta_{n}}-X_{(2^{n}-1)\delta_{n}})$, $Z=\varphi
_{2}(X_{\delta_{n}},X_{2\delta_{n}}-X_{\delta_{n}},\ldots,X_{2^{n}\delta_{n}%
}-X_{(2^{n}-1)\delta_{n}})\in \mathcal{H}^{n}$ with $Y\geq Z$, it is easy to
verify that $Y=(\varphi_{1}\vee \varphi_{2})(X_{\delta_{n}},X_{2\delta_{n}%
}-X_{\delta_{n}},\ldots,X_{2^{n}\delta_{n}}-X_{(2^{n}-1)\delta_{n}})$. Then by
the definition of $\mathbb{\hat{E}}^{n}$ and the monotonicity of
$\mathbb{\hat{E}}_{1}$ and $\mathbb{\hat{E}}_{2}$, we can get $\mathbb{\hat
{E}}^{n}[Y]\geq \mathbb{\hat{E}}^{n}[Z]$. (2) and (3) can be easily obtained by
the definition of $\mathbb{\hat{E}}^{n}$.
\end{proof}

Obviously, $\mathcal{H}^{n}\subset \mathcal{H}^{n+1}$ for each $n\geq1$. We set%
\[
\mathcal{L}=\bigcup \limits_{n\geq1}\mathcal{H}^{n}.
\]
It is easily seen that $\mathcal{L}$ is a subspace of $\mathcal{H}$ such that
if $Y_{1}$,$\dots$,$Y_{m}\in \mathcal{L}$, then $\varphi(Y_{1},\dots,Y_{m}%
)\in \mathcal{L}$ for each $\varphi \in C_{b.Lip}(\mathbb{R}^{m})$. In the
following, we want to define a sublinear expectation $\mathbb{\hat{E}%
}:\mathcal{L}\rightarrow \mathbb{R}$. Unfortunately, $\mathbb{\hat{E}}%
^{n+1}[\cdot]\not =\mathbb{\hat{E}}^{n}[\cdot]$ on $\mathcal{H}^{n}$, because
the order of independence under sublinear expectation space is unchangeable.
But the following lemma will allow us to construct $\mathbb{\hat{E}}$.

\begin{lemma}
\label{le3.3}For each fixed $n\geq1$, let $\mathbb{\hat{F}}_{k}^{n}$, $k\geq
n$, be the distribution of $(X_{\delta_{n}},X_{2\delta_{n}}-X_{\delta_{n}%
},\ldots,X_{2^{n}\delta_{n}}-X_{(2^{n}-1)\delta_{n}})$ under $\mathbb{\hat{E}%
}^{k}$. Then $\{ \mathbb{\hat{F}}_{k}^{n}:k\geq n\}$ is tight.
\end{lemma}

\begin{proof}
For each given $N>1$ and $\varphi \in C_{b.Lip}(\mathbb{R}^{2^{n}\times2d})$
with $I_{\{x:|x|\geq N\}}\leq \varphi \leq I_{\{x:|x|\geq N-1\}}$, we have%
\begin{align*}
\mathbb{\hat{F}}_{k}^{n}[\varphi] &  =\mathbb{\hat{E}}^{k}[\varphi
(X_{\delta_{n}},X_{2\delta_{n}}-X_{\delta_{n}},\ldots,X_{2^{n}\delta_{n}%
}-X_{(2^{n}-1)\delta_{n}})]\\
&  \leq \frac{1}{N-1}\mathbb{\hat{E}}^{k}[|(X_{\delta_{n}},X_{2\delta_{n}%
}-X_{\delta_{n}},\ldots,X_{2^{n}\delta_{n}}-X_{(2^{n}-1)\delta_{n}})|]\\
&  \leq \frac{1}{N-1}\mathbb{\hat{E}}^{k}[\sum_{i=1}^{2^{n}}(|\tilde
{M}_{i\delta_{n}}-\tilde{M}_{(i-1)\delta_{n}}|+|\tilde{N}_{i\delta_{n}}%
-\tilde{N}_{(i-1)\delta_{n}}|)]\\
&  \leq \frac{1}{N-1}(\mathbb{\hat{E}}^{k}[\sum_{i=1}^{2^{n}}|\tilde
{M}_{i\delta_{n}}-\tilde{M}_{(i-1)\delta_{n}}|]+\mathbb{\hat{E}}^{k}%
[\sum_{i=1}^{2^{n}}|\tilde{N}_{i\delta_{n}}-\tilde{N}_{(i-1)\delta_{n}}|])\\
&  =\frac{1}{N-1}(\mathbb{\hat{E}}_{1}[\sum_{i=1}^{2^{n}}|M_{i\delta_{n}%
}-M_{(i-1)\delta_{n}}|]+\mathbb{\hat{E}}_{2}[\sum_{i=1}^{2^{n}}|N_{i\delta
_{n}}-N_{(i-1)\delta_{n}}|]),
\end{align*}
where the last equality in the above formula is due to (3) in Lemma
\ref{le3.2}. Then for each $\varepsilon>0$, we can take an $N_{0}>1$ and
$\varphi_{0}\in C_{b.Lip}(\mathbb{R}^{2^{n}\times2d})$ with $I_{\{x:|x|\geq
N_{0}\}}\leq \varphi_{0}\leq I_{\{x:|x|\geq N_{0}-1\}}$ such that%
\[
\sup_{k\geq n}\mathbb{\hat{F}}_{k}^{n}[\varphi_{0}]\leq \frac{1}{N_{0}%
-1}(\mathbb{\hat{E}}_{1}[\sum_{i=1}^{2^{n}}|M_{i\delta_{n}}-M_{(i-1)\delta
_{n}}|]+\mathbb{\hat{E}}_{2}[\sum_{i=1}^{2^{n}}|N_{i\delta_{n}}-N_{(i-1)\delta
_{n}}|])<\varepsilon.
\]
Thus $\{ \mathbb{\hat{F}}_{k}^{n}:k\geq n\}$ is tight.
\end{proof}

Now we will use this lemma to construct a sublinear expectation $\mathbb{\hat
{E}}:\mathcal{L}\rightarrow \mathbb{R}$.

\begin{lemma}
\label{le3.4}Set $\mathcal{D}=\{i2^{-n}:n\geq1,0\leq i\leq2^{n}\}$. Then there
exists a sublinear expectation $\mathbb{\hat{E}}:\mathcal{L}\rightarrow
\mathbb{R}$ satisfying the following properties:

\begin{description}
\item[(1)] For each $0\leq t_{1}<\cdots<t_{n}$ with $t_{i}\in \mathcal{D}$,
$i\leq n$, $X_{t_{n}}-X_{t_{n-1}}$ is independent from $(X_{t_{1}}%
,\ldots,X_{t_{n-1}})$;

\item[(2)] For each $0\leq t_{1}<\cdots<t_{n}$ with $t_{i}\in \mathcal{D}$,
$i\leq n$, $(\tilde{M}_{t_{1}},\ldots,\tilde{M}_{t_{n}})\overset{d}%
{=}(M_{t_{1}},\ldots,M_{t_{n}})$ and $(\tilde{N}_{t_{1}},\ldots,\tilde
{N}_{t_{n}})\overset{d}{=}(N_{t_{1}},\ldots,N_{t_{n}})$.
\end{description}
\end{lemma}

\begin{proof}
For $n=1$, by Lemma \ref{le3.3}, we know $\{ \mathbb{\hat{F}}_{k}^{1}:k\geq1\}$
is tight. Then, by Theorem \ref{th2.1}, there exists a subsequence
$\{ \mathbb{\hat{F}}_{k_{j}^{1}}^{1}:j\geq1\}$ which converges in distribution,
i.e., for each $\varphi \in C_{b.Lip}(\mathbb{R}^{2\times2d})$, $\{ \mathbb{\hat
{F}}_{k_{j}^{1}}^{1}[\varphi]:j\geq1\}$ is a Cauchy sequence. Note that
$\mathbb{\hat{F}}_{k_{j}^{1}}^{1}[\varphi]=\mathbb{\hat{E}}^{k_{j}^{1}%
}[\varphi(X_{2^{-1}},X_{1}-X_{2^{-1}})]$, then for each $Y\in \mathcal{H}^{1}$,
$\{ \mathbb{\hat{E}}^{k_{j}^{1}}[Y]:j\geq1\}$ is a Cauchy sequence.

For $n=2$, by Lemma \ref{le3.3} and Theorem \ref{th2.1}, we can find a
subsequence $\{k_{j}^{2}:j\geq1\} \subset \{k_{j}^{1}:j\geq1\}$ such that for
each $Y\in \mathcal{H}^{2}$, $\{ \mathbb{\hat{E}}^{k_{j}^{2}}[Y]:j\geq1\}$ is a
Cauchy sequence.

Repeat this process, for each $n\geq2$, we can find a subsequence $\{k_{j}%
^{n}:j\geq1\} \subset \{k_{j}^{n-1}:j\geq1\}$ such that for each $Y\in
\mathcal{H}^{n}$, $\{ \mathbb{\hat{E}}^{k_{j}^{n}}[Y]:j\geq1\}$ is a Cauchy
sequence. Taking the diagonal sequence $\{k_{j}^{j}:j\geq1\}$, then for each
$Y\in \mathcal{L}$, $\{ \mathbb{\hat{E}}^{k_{j}^{j}}[Y]:j\geq1\}$ is a Cauchy
sequence, where $\mathbb{\hat{E}}^{k_{j}^{j}}[Y]=\infty$ if $Y\not \in
\mathcal{H}^{k_{j}^{j}}$. Define%
\[
\mathbb{\hat{E}}[Y]=\lim_{j\rightarrow \infty}\mathbb{\hat{E}}^{k_{j}^{j}%
}[Y]\text{ for each }Y\in \mathcal{L}.
\]
For each $Y$, $Z\in \mathcal{L}$, there exists a $j_{0}$ such that $Y$,
$Z\in \mathcal{H}^{k_{j}^{j}}$ for $j\geq j_{0}$. From this we can easily
deduce that $\mathbb{\hat{E}}$ is a sublinear expectation.

Now we prove that this $\mathbb{\hat{E}}$ satisfies (1) and (2). For each
$0\leq t_{1}<\cdots<t_{n}$ with $t_{i}\in \mathcal{D}$, $i\leq n$, there exists
a $j_{0}$ such that $\varphi(X_{t_{1}},\ldots,X_{t_{n-1}},X_{t_{n}}%
-X_{t_{n-1}})\in \mathcal{H}^{k_{j}^{j}}$ for each $j\geq j_{0}$ and
$\varphi \in C_{b.Lip}(\mathbb{R}^{n\times2d})$. Thus, from (2) in Lemma
\ref{le3.2}, we can get%
\begin{align*}
\mathbb{\hat{E}}[\varphi(X_{t_{1}},\ldots,X_{t_{n-1}},X_{t_{n}}-X_{t_{n-1}})]
&  =\lim_{j\rightarrow \infty}\mathbb{\hat{E}}^{k_{j}^{j}}[\varphi(X_{t_{1}%
},\ldots,X_{t_{n-1}},X_{t_{n}}-X_{t_{n-1}})]\\
&  =\lim_{j\rightarrow \infty}\mathbb{\hat{E}}^{k_{j}^{j}}[\psi_{j}(X_{t_{1}%
},\ldots,X_{t_{n-1}})],
\end{align*}
where $\psi_{j}(x_{1},\ldots,x_{n-1})=\mathbb{\hat{E}}^{k_{j}^{j}}%
[\varphi(x_{1},\ldots,x_{n-1},X_{t_{n}}-X_{t_{n-1}})]$ for each $x_{i}%
\in \mathbb{R}^{2d}$, $i\leq n-1$. Define $\psi(x_{1},\ldots,x_{n-1}%
)=\mathbb{\hat{E}}[\varphi(x_{1},\ldots,x_{n-1},X_{t_{n}}-X_{t_{n-1}})]$ for
each $x_{i}\in \mathbb{R}^{2d}$, $i\leq n-1$. In order to prove (1), we only
need to show that%
\begin{equation}
\lim_{j\rightarrow \infty}\mathbb{\hat{E}}^{k_{j}^{j}}[\psi_{j}(X_{t_{1}%
},\ldots,X_{t_{n-1}})]=\mathbb{\hat{E}}[\psi(X_{t_{1}},\ldots,X_{t_{n-1}%
})].\label{eq31}%
\end{equation}
It is clear that $\psi_{j}$, $j\geq j_{0}$, and $\psi$ are bounded Lipschitz
functions with the common bound $K_{\varphi}$ and the common Lipschitz
constant $L_{\varphi}$, where $K_{\varphi}$ and $L_{\varphi}$ are respective
the bound and Lipschitz constant of $\varphi$. On the other hand, for each
$x=(x_{1},\ldots,x_{n-1})\in \mathbb{R}^{(n-1)\times2d}$, by the definition of
$\mathbb{\hat{E}}$, we can get $\psi_{j}(x)\rightarrow \psi(x)$. Thus, from the
common Lipschitz constant $L_{\varphi}$ and pointwise convergence, we can
easily obtain that $\{ \psi_{j}:j\geq j_{0}\}$ converges uniformly to $\psi$ on
any compact set in $\mathbb{R}^{(n-1)\times2d}$. For each given $N>0$, we have%
\[
|\psi_{j}(x)-\psi(x)|\leq a_{j}+2K_{\varphi}I_{\{x:|x|>N\}}\leq a_{j}%
+\frac{2K_{\varphi}}{N}|x|,
\]
where $a_{j}=\sup_{|x|\leq N}|\psi_{j}(x)-\psi(x)|$. From the uniform
convergence on any compact set, we know $a_{j}\rightarrow0$ as $j\rightarrow
\infty$. Thus%
\begin{align*}
&  |\mathbb{\hat{E}}^{k_{j}^{j}}[\psi_{j}(X_{t_{1}},\ldots,X_{t_{n-1}%
})]-\mathbb{\hat{E}}^{k_{j}^{j}}[\psi(X_{t_{1}},\ldots,X_{t_{n-1}})]|\\
&  \leq \mathbb{\hat{E}}^{k_{j}^{j}}[|\psi_{j}(X_{t_{1}},\ldots,X_{t_{n-1}%
})-\psi(X_{t_{1}},\ldots,X_{t_{n-1}})|]\\
&  \leq \mathbb{\hat{E}}^{k_{j}^{j}}[a_{j}+\frac{2K_{\varphi}}{N}|(X_{t_{1}%
},\ldots,X_{t_{n-1}})|]\\
&  \leq a_{j}+\frac{2K_{\varphi}}{N}(\mathbb{\hat{E}}^{k_{j}^{j}}[|(\tilde
{M}_{t_{1}},\ldots,\tilde{M}_{t_{n-1}})|]+\mathbb{\hat{E}}^{k_{j}^{j}%
}[|(\tilde{N}_{t_{1}},\ldots \tilde{N}_{t_{n-1}})|])\\
&  =a_{j}+\frac{2K_{\varphi}}{N}(\mathbb{\hat{E}}_{1}[|(M_{t_{1}}%
,\ldots,M_{t_{n-1}})|]+\mathbb{\hat{E}}_{2}[|(N_{t_{1}},\ldots N_{t_{n-1}%
})|]),
\end{align*}
where the last equality in the above formula is due to (3) in Lemma
\ref{le3.2}. Note that $\mathbb{\hat{E}}^{k_{j}^{j}}[\psi(X_{t_{1}}%
,\ldots,X_{t_{n-1}})]\rightarrow \mathbb{\hat{E}}[\psi(X_{t_{1}},\ldots
,X_{t_{n-1}})]$, then%
\begin{align*}
&  \limsup_{j\rightarrow \infty}|\mathbb{\hat{E}}^{k_{j}^{j}}[\psi_{j}%
(X_{t_{1}},\ldots,X_{t_{n-1}})]-\mathbb{\hat{E}}[\psi(X_{t_{1}},\ldots
,X_{t_{n-1}})]|\\
&  \leq \frac{2K_{\varphi}}{N}(\mathbb{\hat{E}}_{1}[|(M_{t_{1}},\ldots
,M_{t_{n-1}})|]+\mathbb{\hat{E}}_{2}[|(N_{t_{1}},\ldots N_{t_{n-1}})|]).
\end{align*}
Since $N$ can be arbitrarily large, we get relation (\ref{eq31}). Thus (1) is
obtained. (2) is obvious by the definition of $\mathbb{\hat{E}}$ and (3) in
Lemma \ref{le3.2}.
\end{proof}

\textbf{Proof of Theorem \ref{th3.1}.} We first extend the sublinear
expectation $\mathbb{\hat{E}}:\mathcal{L}\rightarrow \mathbb{R}$ to
$\mathbb{\hat{E}}:\mathcal{H}\rightarrow \mathbb{R}$. Here we still use
$\mathbb{\hat{E}}$ for notation simplicity. For each $\varphi(X_{t_{1}%
},X_{t_{2}}-X_{t_{1}},\ldots,X_{t_{n}}-X_{t_{n-1}})\in \mathcal{H}$ with
$\varphi \in C_{b.Lip}(\mathbb{R}^{n\times2d})$, we can choose $t_{k}^{i}%
\in \mathcal{D}$, $k\leq n$, $i\geq1$, such that $t_{k}^{i}<t_{k+1}^{i}$ and
$t_{k}^{i}\downarrow t_{k}$ as $i\rightarrow \infty$. By assumption (A), we
have%
\begin{align*}
&  |\mathbb{\hat{E}}[\varphi(X_{t_{1}^{i}},X_{t_{2}^{i}}-X_{t_{1}^{i}}%
,\ldots,X_{t_{n}^{i}}-X_{t_{n-1}^{i}})]-\mathbb{\hat{E}}[\varphi(X_{t_{1}^{j}%
},X_{t_{2}^{j}}-X_{t_{1}^{j}},\ldots,X_{t_{n}^{j}}-X_{t_{n-1}^{j}})]|\\
&  \leq L_{\varphi}\mathbb{\hat{E}}[\sum_{k=1}^{n}|X_{t_{k}^{i}}-X_{t_{k}^{j}%
}-X_{t_{k-1}^{i}}+X_{t_{k-1}^{j}}|]\\
&  \leq L_{\varphi}\mathbb{\hat{E}}[\sum_{k=1}^{n}|(X_{t_{k}^{i}}-X_{t_{k}%
})-(X_{t_{k}^{j}}-X_{t_{k}})-(X_{t_{k-1}^{i}}-X_{t_{k-1}})+(X_{t_{k-1}^{j}%
}-X_{t_{k-1}})|]\\
&  \leq2L_{\varphi}\{ \mathbb{\hat{E}}[\sum_{k=1}^{n}(|\tilde{M}_{t_{k}^{i}%
}-\tilde{M}_{t_{k}}|+|\tilde{M}_{t_{k}^{j}}-\tilde{M}_{t_{k}}|)]+\mathbb{\hat
{E}}[\sum_{k=1}^{n}(|\tilde{N}_{t_{k}^{i}}-\tilde{N}_{t_{k}}|+|\tilde
{N}_{t_{k}^{j}}-\tilde{N}_{t_{k}}|)]\} \\
&  =2L_{\varphi}\{ \mathbb{\hat{E}}_{1}[\sum_{k=1}^{n}(|M_{t_{k}^{i}}-M_{t_{k}%
}|+|M_{t_{k}^{j}}-M_{t_{k}}|)]+\mathbb{\hat{E}}_{2}[\sum_{k=1}^{n}%
(|N_{t_{k}^{i}}-N_{t_{k}}|+|N_{t_{k}^{j}}-N_{t_{k}}|)]\} \\
&  \rightarrow0\text{ as }i,j\rightarrow \infty,
\end{align*}
where $L_{\varphi}$ is the Lipschitz constant of $\varphi$ and $t_{0}^{i}=0$
for $i\geq1$. So we can define%
\[
\mathbb{\hat{E}}[\varphi(X_{t_{1}},X_{t_{2}}-X_{t_{1}},\ldots,X_{t_{n}%
}-X_{t_{n-1}})]=\lim_{i\rightarrow \infty}\mathbb{\hat{E}}[\varphi(X_{t_{1}%
^{i}},X_{t_{2}^{i}}-X_{t_{1}^{i}},\ldots,X_{t_{n}^{i}}-X_{t_{n-1}^{i}})].
\]
It is easy to check that the limit does not depend on the choice of $t_{k}%
^{i}$ by using the same estimate as above. On the other hand, if
$\varphi(X_{t_{1}},X_{t_{2}}-X_{t_{1}},\ldots,X_{t_{n}}-X_{t_{n-1}}%
)=\tilde{\varphi}(X_{t_{1}},X_{t_{2}}-X_{t_{1}},\ldots,X_{t_{n}}-X_{t_{n-1}})$
for $\varphi$, $\tilde{\varphi}\in C_{b.Lip}(\mathbb{R}^{n\times2d})$, then%
\begin{align*}
&  |\varphi(X_{t_{1}^{i}},X_{t_{2}^{i}}-X_{t_{1}^{i}},\ldots,X_{t_{n}^{i}%
}-X_{t_{n-1}^{i}})-\tilde{\varphi}(X_{t_{1}^{i}},X_{t_{2}^{i}}-X_{t_{1}^{i}%
},\ldots,X_{t_{n}^{i}}-X_{t_{n-1}^{i}})|\\
&  \leq|\varphi(X_{t_{1}^{i}},X_{t_{2}^{i}}-X_{t_{1}^{i}},\ldots,X_{t_{n}^{i}%
}-X_{t_{n-1}^{i}})-\varphi(X_{t_{1}},X_{t_{2}}-X_{t_{1}},\ldots,X_{t_{n}%
}-X_{t_{n-1}})|\\
&  \  \ +|\tilde{\varphi}(X_{t_{1}},X_{t_{2}}-X_{t_{1}},\ldots,X_{t_{n}%
}-X_{t_{n-1}})-\tilde{\varphi}(X_{t_{1}^{i}},X_{t_{2}^{i}}-X_{t_{1}^{i}%
},\ldots,X_{t_{n}^{i}}-X_{t_{n-1}^{i}})|\\
&  \leq2(L_{\varphi}+L_{\tilde{\varphi}})\sum_{k=1}^{n}(|M_{t_{k}^{i}%
}-M_{t_{k}}|+|N_{t_{k}^{i}}-N_{t_{k}}|).
\end{align*}
Thus, by assumption (A), $\mathbb{\hat{E}}:\mathcal{H}\rightarrow \mathbb{R}$
is well-defined. It is easy to verify that $\mathbb{\hat{E}}$ is a sublinear
expectation. For each $0\leq t_{1}<\cdots<t_{n}\leq1$, we can choose
$t_{k}^{i}\in \mathcal{D}$ as above. By the definition of $\mathbb{\hat{E}}$
and (1) in Lemma \ref{le3.4}, we can get that for each $\varphi \in
C_{b.Lip}(\mathbb{R}^{n\times2d})$,
\begin{align*}
\mathbb{\hat{E}}[\varphi(X_{t_{1}},\ldots,X_{t_{n-1}},X_{t_{n}}-X_{t_{n-1}})]
&  =\lim_{i\rightarrow \infty}\mathbb{\hat{E}}[\varphi(X_{t_{1}^{i}}%
,\ldots,X_{t_{n-1}^{i}},X_{t_{n}^{i}}-X_{t_{n-1}^{i}})]\\
&  =\lim_{i\rightarrow \infty}\mathbb{\hat{E}}[\psi_{i}(X_{t_{1}^{i}}%
,\ldots,X_{t_{n-1}^{i}})],
\end{align*}
where $\psi_{i}(x_{1},\ldots,x_{n-1})=\mathbb{\hat{E}}[\varphi(x_{1}%
,\ldots,x_{n-1},X_{t_{n}^{i}}-X_{t_{n-1}^{i}})]$. Define%
\[
\psi(x_{1},\ldots,x_{n-1})=\mathbb{\hat{E}}[\varphi(x_{1},\ldots
,x_{n-1},X_{t_{n}}-X_{t_{n-1}})].
\]
Then%
\begin{align*}
&  |\psi_{i}(x_{1},\ldots,x_{n-1})-\psi(x_{1},\ldots,x_{n-1})|\\
&  \leq L_{\varphi}\mathbb{\hat{E}}[|X_{t_{n}^{i}}-X_{t_{n-1}^{i}}-X_{t_{n}%
}+X_{t_{n-1}}|]\\
&  \leq L_{\varphi}\{ \mathbb{\hat{E}}_{1}[|M_{t_{n}^{i}}-M_{t_{n}%
}|+|M_{t_{n-1}^{i}}-M_{t_{n-1}}|]+\mathbb{\hat{E}}_{2}[|N_{t_{n}^{i}}%
-N_{t_{n}}|+|N_{t_{n-1}^{i}}-N_{t_{n-1}}|]\},
\end{align*}
which implies%
\begin{align*}
&  |\mathbb{\hat{E}}[\psi_{i}(X_{t_{1}^{i}},\ldots,X_{t_{n-1}^{i}%
})]-\mathbb{\hat{E}}[\psi(X_{t_{1}^{i}},\ldots,X_{t_{n-1}^{i}})]|\\
&  \leq L_{\varphi}\{ \mathbb{\hat{E}}_{1}[|M_{t_{n}^{i}}-M_{t_{n}%
}|+|M_{t_{n-1}^{i}}-M_{t_{n-1}}|]+\mathbb{\hat{E}}_{2}[|N_{t_{n}^{i}}%
-N_{t_{n}}|+|N_{t_{n-1}^{i}}-N_{t_{n-1}}|]\}.
\end{align*}
From this we deduce%
\begin{align*}
\mathbb{\hat{E}}[\varphi(X_{t_{1}},\ldots,X_{t_{n-1}},X_{t_{n}}-X_{t_{n-1}})]
&  =\lim_{i\rightarrow \infty}\mathbb{\hat{E}}[\psi(X_{t_{1}^{i}}%
,\ldots,X_{t_{n-1}^{i}})]\\
&  =\mathbb{\hat{E}}[\psi(X_{t_{1}},\ldots,X_{t_{n-1}})].
\end{align*}
Thus $(X_{t})_{t\in \lbrack0,1]}$ has independent increments. It follows from
(2) in Lemma \ref{le3.4} and assumption (A) that $(\tilde{M}_{t})_{t\in
\lbrack0,1]}\overset{d}{=}(M_{t})_{t\in \lbrack0,1]}$ and $(\tilde{N}%
_{t})_{t\in \lbrack0,1]}\overset{d}{=}(N_{t})_{t\in \lbrack0,1]}$. If
$(M_{t})_{t\in \lbrack0,1]}$ and $(N_{t})_{t\in \lbrack0,1]}$ are two processes
with stationary and independent increments, then from the construction of
$\mathbb{\hat{E}}$, $(X_{t})_{t\in \lbrack0,1]}$ has stationary increments. The
proof is complete. \  $\Box$

In the following, we give an example to calculate $\mathbb{\hat{E}}$.

\begin{example}
Let $\Gamma_{i}$, $i=1,2$, be two given bounded subset in $\mathbb{R}^{d\times
d}$. Define $G_{i}:\mathbb{S}(d)\rightarrow \mathbb{R}$ by%
\[
G_{i}(A)=\frac{1}{2}\sup_{Q\in \Gamma_{i}}\mathrm{tr}[AQQ^{T}]\text{ for each
}A\in \mathbb{S}(d).
\]
Let $(B_{t}^{i})_{t\geq0}$ be two $d$-dimensional $G_{i}$-Brownian motion
defined on sublinear expectation space $(\Omega_{i},\mathcal{H}_{i}%
,\mathbb{\hat{E}}_{i})$, $i=1$, $2$. In the above, we construct a sublinear
expectation space $(\Omega,\mathcal{H},\mathbb{\hat{E}})$ and a $2d$%
-dimensional process $(\tilde{B}_{t})_{t\geq0}=(\tilde{B}_{t}^{1},\tilde
{B}_{t}^{2})_{t\geq0}$ with stationary and independent increments satisfying
$(\tilde{B}_{t}^{1})_{t\geq0}\overset{d}{=}(B_{t}^{1})_{t\geq0}$ and
$(\tilde{B}_{t}^{2})_{t\geq0}\overset{d}{=}(B_{t}^{2})_{t\geq0}$. Since%
\[
\mathbb{\hat{E}}[|\tilde{B}_{t}|^{3}]\leq4(\mathbb{\hat{E}}[|\tilde{B}_{t}%
^{1}|^{3}]+\mathbb{\hat{E}}[|\tilde{B}_{t}^{2}|^{3}])=4(\mathbb{\hat{E}}%
_{1}[|B_{t}^{1}|^{3}]+\mathbb{\hat{E}}_{2}[|B_{t}^{2}|^{3}])=4Ct^{\frac{3}{2}%
},
\]
where $C=\mathbb{\hat{E}}_{1}[|B_{1}^{1}|^{3}]+\mathbb{\hat{E}}_{2}[|B_{1}%
^{2}|^{3}]$, by Theorem 1.6 in Chapter III in \cite{P7}, we can obtain that
$(\tilde{B}_{t})_{t\geq0}$ is a $G$-Brownian motion with%
\[
G(A)=\frac{1}{2}\mathbb{\hat{E}}[\langle A\tilde{B}_{1},\tilde{B}_{1}%
\rangle]\text{ for each }A=\left[
\begin{array}
[c]{cc}%
A_{1} & D\\
D^{T} & A_{2}%
\end{array}
\right]  \in \mathbb{S}(2d).
\]
By our construction, it is easy to check that $\mathbb{\hat{E}}^{n}[\langle
D\tilde{B}_{1}^{2},\tilde{B}_{1}^{1}\rangle]=\mathbb{\hat{E}}^{n}[-\langle
D\tilde{B}_{1}^{2},\tilde{B}_{1}^{1}\rangle]=0$ for each $n\geq1$. Thus
$\mathbb{\hat{E}}[\langle D\tilde{B}_{1}^{2},\tilde{B}_{1}^{1}\rangle
]=\mathbb{\hat{E}}[-\langle D\tilde{B}_{1}^{2},\tilde{B}_{1}^{1}\rangle]=0$.
By subadditivity, we can get $\mathbb{\hat{E}}[\langle A\tilde{B}_{1}%
,\tilde{B}_{1}\rangle]=\mathbb{\hat{E}}[\langle A_{1}\tilde{B}_{1}^{1}%
,\tilde{B}_{1}^{1}\rangle+\langle A_{2}\tilde{B}_{1}^{2},\tilde{B}_{1}%
^{2}\rangle]$. Furthermore,
\[
\mathbb{\hat{E}}^{n}[\langle A_{1}\tilde{B}_{1}^{1},\tilde{B}_{1}^{1}%
\rangle+\langle A_{2}\tilde{B}_{1}^{2},\tilde{B}_{1}^{2}\rangle]=\mathbb{\hat
{E}}_{1}[\langle A_{1}B_{1}^{1},B_{1}^{1}\rangle]+\mathbb{\hat{E}}_{2}[\langle
A_{2}B_{1}^{2},B_{1}^{2}\rangle]
\]
for each $n\geq1$ by our construction. Thus $G(A)=G_{1}(A_{1})+G_{2}(A_{2})$.
\end{example}

\end{document}